\documentclass{CSML}

\pdfoutput=1

\usepackage{lastpage}

\def\dOi{13(3:3)2017}
\lmcsheading%
{\dOi}
{1--\pageref{LastPage}}
{}
{}
{Mar.~14, 2016}
{Jul.~\phantom06, 2017}
{}

\usepackage{hyperref}\hypersetup{hidelinks}

\newcommand{\dotminus}{\buildrel\textstyle.\over{\hbox{\vrule height2pt depth0pt width0pt}{\smash-}}}

\theoremstyle{plain}

\begin{document}

\title[Characterization Theorem for the Conditionally Computable Real Functions]{Characterization Theorem for the Conditionally Computable Real Functions}

\author[I.~Georgiev]{Ivan Georgiev}	
\address{Prof. D-r Asen Zlatarov University, bul. Prof. Yakimov 1, Burgas, Bulgaria}	
\email{ivandg@btu.bg}  

\subjclass{[\textbf{Theory of Computation}]: Models of Computation; [\textbf{Theory of Computation}]: Complexity Measures and Classes}
\titlecomment{This paper is based on an informal presentation with the same title that I gave at the CIE conference in 2014, Budapest and on another talk, presented in a conference in 2014, Gyolechitsa, dedicated to the memory of prof. Ivan Soskov.}


\begin{abstract}
  \noindent The class of uniformly computable real functions with respect to a small subrecursive class of operators computes the elementary functions of calculus, restricted to compact subsets of their domains. The class of conditionally computable real functions with respect to the same class of operators is a proper extension of the class of uniformly computable real functions and it computes the elementary functions of calculus on their whole domains. The definition of both classes relies on certain transformations of infinitistic names of real numbers. In the present paper, the conditional computability of real functions is characterized in the spirit of Tent and Ziegler, avoiding the use of infinitistic names.
\end{abstract}

\maketitle

\section*{Introduction}\label{S:one}

This paper is in the field of computable analysis, where mathematical analysis meets classical computability theory. Some of the basic objects, which are studied in computable analysis are those real numbers and real functions, which can be computed using algorithms.

Alan Turing gave in \cite{turing} the first definition for the notion of computable real number -- this is a real number, whose decimal representation can be computed by a Turing machine. It is well-known that several other representations of real numbers lead to an equivalent notion, although the use of fast converging Cauchy sequences of rational numbers seems most natural when developing computable analysis.

Concerning computability of real functions, the most popular approach is TTE (type-2 theory of effectivity), which is built upon ideas from Grzegorczyk \cite{grzcomp} and Lacombe \cite{lacombe}. A naming system for the real numbers is chosen and then the real function is modelled as transforming arbitrary names of the arguments into a name of the value. This transformation can be realized, for example, by using Type-2 machines (a natural extension of Turing machines), as is described in Weihrauch's book \cite{weihrauch}, which is a popular introductory book on the subject.

The motivating question for our research is the following:

\emph{How does the restriction of the generality of the computable processes affect computable analysis?}

In other words, we are interested in studying the connection between computable analysis and complexity theory. This question is posed as an open problem number 1 in the monography \cite{pourelrich} of Pour-El and Richards. The restriction to the popular classes $P, NP, EXP,$ etc. from discrete complexity theory is relatively well-studied (for example, in Ko's book \cite{keriko}). But this is not the case with subrecursive hierarchies, such as Grzegorczyk's hierarchy of primitive recursive functions.

In the present paper we consider two notions for relative computability of real functions, which arose in the study of the subrecursive complexity of the elementary functions of calculus. The motivation behind their definition is to find a small subrecursive class of operators, such that all elementary functions
of calculus are uniformly or conditionally computable with respect to this class.

The first notion is \textit{uniform computability} of real functions with respect to a class of operators, which is a relativized version of Grzegorczyk's notion from \cite{grzcomp}. These operators model the action of the real function through a properly chosen naming system for the real numbers. Theorem 2 from Section 4 in \cite{grzcomp} shows that any real function, which is uniformly computable with respect to a class of computable operators, is uniformly continuous on the bounded subsets of its domain. It follows that this uniform notion is not suitable for computing the reciprocal and the logarithmic function on their whole domains.

This is the reason to consider the second notion, which is \textit{conditional computability} of real functions with respect to a class of operators. The additional feature, compared to the uniform computability, is that we allow the computation of the name of the real function's value to depend on a natural parameter $s$. The value of $s$ can be found by means of a search until some unary total function in the natural numbers reaches value $0$ for argument $s$. This function depends on the names of the arguments of the real function in a way that can be expressed through the class of operators.

In the paper \cite{tz}, the authors Tent and Ziegler follow a similar line of research and also consider two kinds of relative computability of real functions. But their approach does not rely on operators and names of real numbers and they work more directly with rational approximations of the arguments and the value of the real function.

For a class $\mathcal{F}$ of total functions in the natural numbers, satisfying certain natural properties, Tent and Ziegler define in \cite{tz} what it means for a real function with an open domain to be \textit{in $\mathcal{F}$} and to be \textit{uniformly in $\mathcal{F}$}. The first notion is more general than the second one. The difference between these two notions is similar to the difference between the uniform and the conditional computability -- for the broader notion, the approximation process can utilize an additional natural parameter, but its description uses the distance to the complement of the domain of the real function and does not use the class $\mathcal{F}$. Thus the means, provided by the class $\mathcal{F}$, might not be sufficient for computing the value of the parameter, if the domain of the real function is too complicated. As we note below, this is indicated by the fact that there exist incomputable real functions, which are in the class of all recursive functions.

It turns out that for the above classes of functions $\mathcal{F}$, the property of a real function to be uniformly in $\mathcal{F}$ is equivalent to the uniform computability of the real function with respect to a very natural class of operators -- the class of $\mathcal{F}$-substitutional operators, defined in Section 2.2 of \cite{m2}. This is the characterization theorem of Skordev, proven in \cite{chart}, and its more general version from \cite{chartgen}.

But this resemblance does not generalize for the broader notions. For example, let $\mathcal{L}^2$ be the class of lower elementary functions (this is the smallest class of total functions in the natural numbers, which contains the initial functions defined below and is closed under substitution and bounded summation). As noted in the end of the last section of \cite{skge}, the class of real functions in $\mathcal{L}^2$ is not closed under composition, it does not contain all elementary functions of calculus and it also contains incomputable real functions. On the other hand, none of these three unnatural properties is possessed by the class of conditionally computable real functions with respect to the $\mathcal{L}^2$-substitutional operators.

Our main purpose is to present another definition for conditional computability, in the style of Tent and Ziegler, for which we can generalize the characterization theorem. Of course, this definition will not possess the above-mentioned unwanted features. The proof of the generalization will use substantially the ideas from \cite{chartgen}.

\subsection*{On notation}
Throughout the paper $\mathbb{N}$ is the set of all natural numbers (the non-negative integers) and $\mathbb{R}$
is the set of all real numbers. For $m \in \mathbb{N}$ we denote by $\mathcal{T}_m$ the set \mbox{$\{f\,|\,f:\mathbb{N}^m\to\mathbb{N}\}$} of all
$m$-argument total functions in $\mathbb{N}$. Let $\mathcal{T}$ be the set of all total functions in $\mathbb{N}$, $\mathcal{T} = \bigcup_m\mathcal{T}_m$. Unless otherwise specified, a \emph{function} means a function from $\mathcal{T}$. We make distinction between two sorts of variables: $a, b, e, f, g, h$ (possibly with indices) range over functions from $\mathcal{T}_1$ and $x, y, z, s, t, k, m, n, p, q, r$ (possibly with indices) range over numbers from $\mathbb{N}$.
We also use shorthand notation $\vec{f},\vec{g},\vec{h}$ for tuples of functions from $\mathcal{T}_1$ and $\vec{x},\vec{y},\vec{z},\vec{s}$ for tuples of numbers from $\mathbb{N}$. The size of the tuples will always be clear from the context. 

For any $s\in\mathbb{N}$, we denote by $\hat{s}$ the unary constant function with value $s$, $\hat{s} = \lambda x.s$.

The \emph{initial functions} are the projections $\lambda x_1\ldots x_n.x_k$ for all $n,k$ with $1 \leq k \leq n$, the successor function $\lambda x.x+1$, the multiplication function $\lambda xy.xy$, the modified subtraction function $\lambda xy.x \dotminus y = \lambda xy.max(x-y,0)$ and the quotient function $\lambda xy.\left\lfloor\frac{x}{y+1}\right\rfloor$.

For any $k$ and any function $f\in\mathcal{T}_{k+1}$, we define the function $g\in\mathcal{T}_{k+1}$ by
\[  g(\vec{x},y) = \left\{ \begin{array}{rl}
          z & \mbox{if }\; z \leq y,\;\; f(\vec{x},z) = 0 \;\mbox{ and }\; \forall t < z [f(\vec{x},t) \neq 0], \\
          y+1 & \mbox{if }\; \forall t \leq y [f(\vec{x},t) \neq 0].
        \end{array}
\right.\]
We denote $g(\vec{x},y) = \mu_{z \leq y}[f(\vec{x},z) = 0]$ and we say that $g$ is produced from $f$ by \emph{bounded minimization}.

The function $g\in\mathcal{T}_m$ \emph{majorizes} the function $f\in\mathcal{T}_m$ (or $f$ \emph{is majorized by} $g$) if for all $\vec{x}\in\mathbb{N}^m$ we have $f(\vec{x}) \leq g(\vec{x})$.

For natural numbers $n$, the mappings $F : \mathcal{T}_1^n\to\mathcal{T}_1$ will be called \emph{$n$-operators}. An $operator$ is an $n$-operator for some $n$. We generally denote operators by capital letters $E, F, G, H$ (possibly with indices).

A triple of functions $(f,g,h)\in\mathcal{T}_1^3$ will be said to \emph{name} a real number $\xi$, if
$$ \left|\frac{f(n)-g(n)}{h(n)+1} - \xi\right| < \frac{1}{n+1} $$
for all $n\in\mathbb{N}$.

\section{Acceptable pairs}\label{S:two}

A central notion in \cite{chartgen} is the notion of acceptable pair. For our purposes we need to use a definition, which is a little stronger.

\begin{defi}\label{accpair}
Let $\mathcal{F} \subseteq \mathcal{T}$ be a class of functions and $\mathbf{O}$ be a class of operators. The pair $(\mathcal{F},\mathbf{O})$ will be called \emph{acceptable}, if the following conditions hold:
\begin{enumerate}
 \item The initial functions belong to $\mathcal{F}$.
 \item The class $\mathcal{F}$ is closed under substitution and bounded minimization.
 \item All operators in $\mathbf{O}$ are continuous, that is for any $n$-operator $F \in \mathbf{O}$, functions $f_1,\ldots,f_n\in\mathcal{T}_1$ and $x$, there exists a natural number $z$, such that 
$$ F(\vec{g})(x) = F(\vec{f})(x), $$ 
whenever $g_1,\ldots,g_n\in\mathcal{T}_1$ and $g_1(t) = f_1(t),\,\ldots,\,g_n(t) = f_n(t)$ for all $t \leq z$.
 \item For any $n$, the $n$-operator $F$ defined by $F(\vec{f})(x) = x$ belongs to $\mathbf{O}$.
 \item For any $n$ and $k\in\{1,\ldots,n\}$, if the $n$-operator $F_0$ belongs to $\mathbf{O}$, then so does the $n$-operator $F$ defined by
$$ F(\vec{f})(x) = f_k(F_0(\vec{f})(x)). $$
 \item For any $m,n$ and function $a\in\mathcal{T}_m\cap\mathcal{F}$, if $F_1,\ldots,F_m$ are $n$-operators belonging to $\mathbf{O}$, then so is the $n$-operator $F$ defined by
$$ F(\vec{f})(x) = a(F_1(\vec{f})(x),\ldots,F_m(\vec{f})(x)). $$
 \item The class $\mathbf{O}$ is closed under composition of operators, that is if $F$ is a $k$-operator from $\mathbf{O}$ and $G_1,\ldots,G_k$ are $n$-operators all belonging to $\mathbf{O}$, then the $n$-operator $H$ defined by
 $$ H(\vec{f}) = F(G_1(\vec{f}),\ldots,G_k(\vec{f})) $$
also belongs to $\mathbf{O}$.
 \item For any $m,n$, whenever $f_1,\ldots,f_n \in \mathcal{T}_{m+1}\cap\mathcal{F}$ and $F$ is an $n$-operator from $\mathbf{O}$, the function $a\in\mathcal{T}_{m+1}$ defined by
\begin{align*}
 a(\vec{s},x) = F(\lambda t.f_1(\vec{s},t), \ldots, \lambda t.f_n(\vec{s},t))(x)
\end{align*}
belongs to $\mathcal{F}$.
 \item (\emph{uniformity condition}) For any $n$ and $n$-operator $F \in \mathbf{O}$, there exists a $1$-operator $\Omega \in \mathbf{O}$, such that for any $x \in \mathbb{N}$ and any monotonically increasing $g \in \mathcal{T}_1$, if the unary functions $f_1,\ldots,f_n,f_1',\ldots,f_n'$ are majorized by $g$ and 
$$f_1(t) = f_1'(t),\ldots,f_n(t) = f_n'(t) $$ for all $t \leq \Omega(g)(x)$, then
$$ F(f_1,\ldots,f_n)(x) = F(f_1',\ldots,f_n')(x). $$
\end{enumerate}
\end{defi}

\noindent The differences from the definition in \cite{chartgen} are the following: we assume additionally that the class of functions $\mathcal{F}$ is closed under bounded minimization, the class $\mathbf{O}$ consists of continuous operators only and it is closed under composition of operators. The important thing is that if $(\mathcal{F},\mathbf{O})$ is acceptable according to \ref{accpair}, then $(\mathcal{F},\mathbf{O})$ is also acceptable in the sense of the definition from \cite{chartgen}. This allows us to use the results from \cite{chartgen}.

The least class of functions that satisfies conditions (1) and (2) from Definition \ref{accpair} is the class $\mathcal{M}^2$. It is not hard to see that a function $f$ belongs to $\mathcal{M}^2$ if and only if the graph of $f$ is $\Delta_0$-definable and $f$ is majorized by a polynomial. It is known that $\mathcal{M}^2 \subseteq \mathcal{L}^2$, but whether this inclusion is proper is an open problem.

Let $\mathcal{F}$ be a class, satisfying (1) and (2) from Definition \ref{accpair}. Of course, $\mathcal{M}^2 \subseteq \mathcal{F}$. It is easy to see that $\mathcal{F}$ is closed under bounded minimum and bounded maximum operators, that is for any function $f \in \mathcal{T}_{k+1}\cap\mathcal{F}$, the functions
$$ \lambda\vec{x}y.\min_{z \leq y}f(\vec{x},z),\;\;\; \lambda\vec{x}y.\max_{z \leq y}f(\vec{x},z) $$
also belong to $\mathcal{F}$.

For an arbitrary class of functions $\mathcal{F}$, we denote by $\mathbf{O}_\mathcal{F}$ the least class of operators, which satisfies conditions (4), (5) and (6) from Definition \ref{accpair}. This class coincides with the class of $\mathcal{F}$-substitutional operators, defined in \cite{m2}. An $n$-operator $F$ is $\mathcal{F}$-substitutional if and only if there exists a term for $F(f_1,\ldots,f_n)(x)$, built from the variable $x$, the function symbols $f_1,\ldots,f_n$ and some function symbols for functions from $\mathcal{F}$. Of course, if $(\mathcal{F},\mathbf{O})$ is an acceptable pair, then $\mathbf{O}_\mathcal{F} \subseteq \mathbf{O}$.

Examples for acceptable pairs are given in \cite{chartgen} (it is easy to see that they satisfy the stronger Definition \ref{accpair}). These are the pairs $(\mathcal{F},\mathbf{O})$, where:
\begin{itemize}
 \item $\mathcal{F}$ is the class of all recursive functions and $\mathbf{O}$ is the class of all computable operators;
 \item $\mathcal{F}$ is the class of all primitive recursive functions and $\mathbf{O}$ is the class of all primitive recursive operators;
 \item $\mathcal{F}$ is the class of all functions, which are elementary in Kalm\'ar's sense and $\mathbf{O}$ is the class of all elementary operators.
\end{itemize}

Another family of acceptable pairs with least possible second component is given by the following proposition.

\begin{prop}\label{propac} Let $\mathcal{F}$ be a class of functions, which satisfies conditions (1) and (2) from Definition \ref{accpair}. Then the pair $(\mathcal{F},\mathbf{O}_\mathcal{F})$ is acceptable.
\end{prop}
\begin{proof}
Since $\mathcal{F}$ is closed under bounded maximum, every function in $\mathcal{F}$ is majorized by a function in $\mathcal{F}$, which is increasing with respect to all of its arguments. From Theorem 1 in \cite{chartgen}, the pair $(\mathcal{F},\mathbf{O}_\mathcal{F})$ is acceptable in the sense of \cite{chartgen}. It remains to note that straightforward proofs by induction on $F$ show that every operator $F \in \mathbf{O}_\mathcal{F}$ is continuous and that condition (7) from \ref{accpair} holds for the class $\mathbf{O}_\mathcal{F}$.
\end{proof}

As noted in \cite{chartgen}, for the three examples above we have $\mathbf{O}_\mathcal{F} \subseteq \mathbf{O}$, but $\mathbf{O}_\mathcal{F} \neq \mathbf{O}$. Thus the class $\mathcal{F}$ in these examples is the first component of at least two different acceptable pairs. On the other hand, a remark in the end of Section 3 in \cite{chartgen} shows that if $(\mathcal{F}_1,\mathbf{O})$ and $(\mathcal{F}_2,\mathbf{O})$ are acceptable pairs, then $\mathcal{F}_1 = \mathcal{F}_2$.

\section{Uniform computability of real functions}

The formal notion for computability of real functions is based on operators, acting on names of real numbers. We use the definition for computing system from \cite{m2}.

\begin{defi}\label{compsys}
Let $k \in \mathbb{N}$ and $\theta : D \rightarrow \mathbb{R}$, where $D \subseteq \mathbb{R}^k$, be a real function.
The triple $(F,G,H)$, where $F, G, H$ are $3k$-operators, is called a \emph{computing system} for $\theta$ if for all $(\xi_1,\xi_2,\ldots,\xi_k) \in D$ and triples $(f_i,g_i,h_i)$ that name $\xi_i$ for $i = 1, 2, \ldots, k$, the triple
$$ (F(f_1,g_1,h_1,f_2,g_2,h_2,\ldots,f_k,g_k,h_k),$$
 $$ G(f_1,g_1,h_1,f_2,g_2,h_2,\ldots,f_k,g_k,h_k),$$
$$ H(f_1,g_1,h_1,f_2,g_2,h_2,\ldots,f_k,g_k,h_k)) $$
names the real number $\theta(\xi_1,\xi_2,\ldots,\xi_k).$
\end{defi}

\begin{defi}\label{uniform}
Let $\mathbf{O}$ be a class of operators. A real function $\theta$ will be called \emph{uniformly} $\mathbf{O}$-\emph{computable}, if there exists a computing system $(F,G,H)$ for $\theta$, such that $F, G, H \in \mathbf{O}$.
\end{defi}

If $\mathbf{O}$ is the class of all computable operators, then a real function $\theta$ is uniformly $\mathbf{O}$-computable if and only if it is computable in the sense of Grzegorczyk from \cite{grzcomp}. But this is not the most general notion for a computable real function. In order for Definition \ref{compsys} to make sense, the operators $F, G, H$ must be defined only on tuples of names for the arguments of the real function $\theta$. By allowing partial operators we obtain the generally accepted notion for a computable real function, which we will call \emph{computability in the extended sense}.

Let $\mathcal{F}$ be a class of recursive functions. Then it is easy to see that all operators in $\mathbf{O}_\mathcal{F}$ are computable and it follows that the uniformly $\mathbf{O}_\mathcal{F}$-computable real functions are computable in Grzegorczyk's sense. In fact, these are exactly the uniformly $\mathcal{F}$-computable real functions, considered in \cite{m2}.

The results from \cite{m2} imply that all elementary functions of calculus are uniformly $\mathbf{O}_{\mathcal{M}^2}$-computable, but restricted to compact subsets of their domains. One general reason for this restriction was already noted in the introduction -- the reciprocal and the logarithmic functions cannot be computable in Grzegorczyk's sense on $(0,1)$, since they are not uniformly continuous there. Another reason is indicated in Section 2.2 in \cite{m2} -- namely that any uniformly $\mathbf{O}_{\mathcal{M}^2}$-computable real function is bounded by some polynomial (this is due to the fact that the functions in $\mathcal{M}^2$ have polynomial growth). Thus the exponential function cannot be uniformly $\mathbf{O}_{\mathcal{M}^2}$-computable on its whole domain.

Next we present the definition for uniform computability in the style of Tent and Ziegler from \cite{tz}.

\begin{defi}\label{tzuniform}
Let $\mathcal{F} \subseteq \mathcal{T}$ be a class of functions. The real function $\theta : D \rightarrow \mathbb{R}$, where $D \subseteq \mathbb{R}^k$ for some $k\in\mathbb{N}$, will be called \emph{TZ-style uniformly $\mathcal{F}$-computable}, if there exist functions $d \in \mathcal{T}_1 \cap \mathcal{F}$ and $f,g,h \in \mathcal{T}_{3k+1} \cap \mathcal{F}$, such that for all $(\xi_1,\ldots,\xi_k) \in D$ and $p_1, q_1, r_1, \ldots, p_k, q_k, r_k, t \in \mathbb{N}$, the inequalities
$$ |\xi_i| \leq t+1,\;\;\; \left| \frac{p_i - q_i}{r_i + 1} - \xi_i \right| < \frac{1}{d(t)+1}\;\; (i = 1, \ldots, k) $$
imply that the numbers
$$ p = f(p_1,q_1,r_1,\ldots,p_k,q_k,r_k,t),\;\;\;\; q = g(p_1,q_1,r_1,\ldots,p_k,q_k,r_k,t), $$
$$ r = h(p_1,q_1,r_1,\ldots,p_k,q_k,r_k,t) $$
satisfy the inequality
$$ \left| \frac{p-q}{r+1} - \theta(\xi_1,\ldots,\xi_k) \right| < \frac{1}{t+1}. $$
\end{defi}

If $\mathcal{F}$ is a good class in the sense of \cite{tz} and $\theta : D \rightarrow \mathbb{R}$ is a real function with an open domain $D$, then $\theta$ is TZ-style uniformly $\mathcal{F}$-computable if and only if it is uniformly in $\mathcal{F}$ in the sense of \cite{tz}.

We are now ready to formulate the general characterization theorem of Skordev.

\begin{thm}[Skordev, \cite{chartgen}] Let $(\mathcal{F},\mathbf{O})$ be an acceptable pair, $k$ be a natural number and $\theta : D \rightarrow \mathbb{R}$, where $D \subseteq \mathbb{R}^k$ be a real function. Then $\theta$ is uniformly $\mathbf{O}$-computable if and only if $\theta$ is TZ-style uniformly $\mathcal{F}$-computable.
\end{thm}

\begin{cor}[Skordev, \cite{chart}] Let $\mathcal{F}$ be a class of functions, satisfying conditions (1) and (2) from Definition \ref{accpair}. Then a real function is uniformly $\mathbf{O}_\mathcal{F}$-computable if and only if it is TZ-style uniformly $\mathcal{F}$-computable.
\end{cor}
\begin{proof} We apply Proposition \ref{propac}. \end{proof}
\begin{cor}\label{coruni} For an acceptable pair $(\mathcal{F},\bf{O})$ and a real function $\theta$, the following are equivalent:
	\begin{itemize}
	 \item $\theta$ is uniformly $\mathbf{O}$-computable,
	 \item $\theta$ is TZ-style uniformly $\mathcal{F}$-computable,
	 \item $\theta$ is uniformly $\mathbf{O}_\mathcal{F}$-computable.
  \end{itemize}
	\end{cor}

\section{Conditional computability of real functions}

In this section we present the notion of conditionally computable real function.

\begin{defi}\label{defcondo}
Let $\mathbf{O}$ be a class of operators, $k$ be a natural number and $\theta:D\to\mathbb{R}$ be a real function, where $D\subseteq\mathbb{R}^k$. Then $\theta$ will be called \emph{conditionally $\mathbf{O}$-computable}, if there exist $3k$-operator $E$ and $(3k+1)$-operators $F,G,H$, all belonging to $\mathbf{O}$, such that whenever $(\xi_1,\ldots,\xi_k)\in D$ and $(f_1,g_1,h_1),\,\ldots,\,(f_k,g_k,h_k)$ are triples that name $\xi_1,\,\ldots,\,\xi_k$, respectively, the following holds:
\begin{enumerate}
\item There exists a natural number $s$, satisfying the equality
\begin{equation}\label{E2}
E(f_1,g_1,h_1,\ldots,f_k,g_k,h_k)(s) = 0.
\end{equation}
\item For any natural number $s$, which satisfies (\ref{E2}), the triple
\begin{align*}
(F(f_1,g_1,h_1,\ldots,f_k,g_k,h_k,\widehat{s}),\\G(f_1,g_1,h_1,\ldots,f_k,g_k,h_k,\widehat{s}),\\H(f_1,g_1,h_1,\ldots,f_k,g_k,h_k,\widehat{s}))
\end{align*}
names the real number $\theta(\xi_1,\ldots,\xi_k)$.
\end{enumerate}
\end{defi}

\noindent This is not the exact form of the original notion, that is used in Definition 2 in \cite{skge}. The dependence on the value of $s$ in the approximation process in \cite{skge} is realized by adding $s$ as a natural argument of the functions, which are the values of the operators $F, G, H$. In the present definition, we use the approach from Definition 3.2 in \cite{gesk}, where the constant function $\widehat{s}$ is added to the arguments of $F, G, H$. This is the reason why the operators $F, G, H$ have one argument more than the operator $E$, which is indicated in the definitions below by denoting this extra argument as $a$.

If $\mathbf{O}$ consists of computable operators, then the conditionally $\mathbf{O}$-computable real functions are computable in the extended sense, but generally not in Grzegorczyk's sense. The search for a value of $s$, which satisfies equality (\ref{E2}) might not be successful, if the arguments of $E$ are not tuples of names for real numbers from the domain $D$ of $\theta$.

Let $\mathcal{F}$ be a class of recursive functions. Then the operators in $\mathbf{O}_\mathcal{F}$ are computable and therefore the conditionally $\mathbf{O}_\mathcal{F}$-computable real functions are computable in the extended sense. In fact, these are exactly the conditionally $\mathcal{F}$-computable real functions, firstly considered in \cite{skge}, as follows from the $m = 1$ case of Lemma 2.4 in \cite{gesk}.

\begin{exa} Let $\mathbf{O}$ be a class of operators, which is the second component of some acceptable pair $(\mathcal{F},\mathbf{O})$. Then all uniformly $\mathbf{O}$-computable real functions are conditionally $\mathbf{O}$-computable. Indeed, let $\theta$ be a real function, $\theta : D \to \mathbb{R}, D \subseteq \mathbb{R}^k, k \in \mathbb{N}$ and $(F^\circ,G^\circ,H^\circ)$ be a computing system for $\theta$, which consists of operators, belonging to $\mathbf{O}$. Then we can satisfy the requirements of Definition \ref{defcondo} through the operators $E, F, G, H$, defined by
\begin{align*}
E(f_1,g_1,h_1,\ldots,f_k,g_k,h_k) = &\,\mathrm{id}_\mathbb{N},\\
F(f_1,g_1,h_1,\ldots,f_k,g_k,h_k,a)\! = &\,F^\circ(f_1,g_1,h_1,\ldots,f_k,g_k,h_k),\\
G(f_1,g_1,h_1,\ldots,f_k,g_k,h_k,a)\! = &\,G^\circ(f_1,g_1,h_1,\ldots,f_k,g_k,h_k),\\
H(f_1,g_1,h_1,\ldots,f_k,g_k,h_k,a)\! = &\,H^\circ(f_1,g_1,h_1,\ldots,f_k,g_k,h_k).
\end{align*}
The operator $E$ belongs to $\mathbf{O}$ from condition (4) of \ref{accpair}. Since the operators $F^\circ,G^\circ,H^\circ$ belong to $\mathbf{O}$, conditions (4), (5) and (7) from \ref{accpair} imply that $F,G,H$ also belong to $\mathbf{O}$.
\end{exa}

The connections between uniform $\mathbf{O}$-computability and conditional $\mathbf{O}$-computability of real functions are further investigated in the paper \cite{gesk}. For classes of operators $\mathbf{O}$, which satisfy certain natural properties, it is shown in \cite{gesk} that:
\begin{itemize}
 \item substitution of real functions preserves conditional $\mathbf{O}$-computability;
 \item all conditionally $\mathbf{O}$-computable real functions are locally uniformly $\mathbf{O}$-computable;
 \item the conditionally $\mathbf{O}$-computable real functions with compact domains are uniformly $\mathbf{O}$-computable.
\end{itemize}
All elementary functions of calculus are conditionally $\mathbf{O}_{\mathcal{M}^2}$-computable on their whole domains (Corollary 1 in \cite{skge}). Roughly speaking, for the reciprocal and the logarithmic function, the parameter $s$ can be used to isolate the argument from $0$. For the exponential function, the parameter $s$ can be used to compute an upper bound of its value, using the fact that the graph of $\lambda xy.x^y$ (belonging to $\mathcal{T}_2$) is $\Delta_0$-definable.

\section{Some preliminary results}

We will need some results from \cite{skper}, which allow choosing a special kind of operators for witnesses in the definition for conditional computability.

For a natural number $k$ and a $k$-tuple $(\xi_1,\ldots,\xi_k) \in \mathbb{R}^k$ we denote by $\mathbb{A}_{\xi_1,\ldots,\xi_k}$ the set of all $2k$-tuples of unary functions $(f_1,g_1,\ldots,f_k,g_k)$, such that $f_1(n).g_1(n) = \ldots = f_k(n).g_k(n) = 0$ for all $n \in \mathbb{N}$ and the triples $(f_1,g_1,\mathrm{id}_\mathbb{N}), \ldots, (f_k,g_k,\mathrm{id}_\mathbb{N})$ name the real numbers $\xi_1,\ldots,\xi_k$, respectively.

Under the same assumptions for all $n\in\mathbb{N}$ let us denote by $\mathbb{A}_{\xi_1,\ldots,\xi_k}^{[n]}$ the set of all $2k$-tuples of natural numbers $(x_1,y_1,\ldots,x_k,y_k)$, such that the following conditions hold:
$$ \left|x_1 - y_1 - (n+1)\xi_1\right| < 1,\;\;\ldots,\;\;\left|x_k - y_k - (n+1)\xi_k\right| < 1,$$
$$ x_1.y_1 = \ldots = x_k.y_k = 0. $$

We can identify the set $\mathbb{A}_{\xi_1,\ldots,\xi_k}$ with the Cartesian product
$$ \mathbb{A}_{\xi_1,\ldots,\xi_k}^{[0]} \times \mathbb{A}_{\xi_1,\ldots,\xi_k}^{[1]} \times \mathbb{A}_{\xi_1,\ldots,\xi_k}^{[2]} \times \ldots $$
through the bijection $\mathbb{F}$, which maps the $2k$-tuple $(f_1,g_1,\ldots,f_k,g_k) \in \mathbb{A}_{\xi_1,\ldots,\xi_k}$ to the unary function $t = \mathbb{F}(f_1,g_1,\ldots,f_k,g_k)$, defined by $t(n) = (f_1(n),g_1(n),\ldots,f_k(n),g_k(n))$, which is easily seen to belong to the Cartesian product.

For all $n\in\mathbb{N}$ the set $\mathbb{A}_{\xi_1,\ldots,\xi_k}^{[n]}$ contains at most $2^k$ elements and therefore it is compact in any topology. We endow every $\mathbb{A}_{\xi_1,\ldots,\xi_k}^{[n]}$ with discrete topology and in the Cartesian product we introduce the product topology. According to Tychonoff's theorem the Cartesian product is compact. Using the bijection $\mathbb{F}$ we transfer the topology from the Cartesian product into $\mathbb{A}_{\xi_1,\ldots,\xi_k}$, that is $U$ is open in $\mathbb{A}_{\xi_1,\ldots,\xi_k}$ if and only if $\mathbb{F}[U]$ is open in the Cartesian product. Of course, $\mathbb{F}$ becomes a homeomorphism and therefore $\mathbb{A}_{\xi_1,\ldots,\xi_k}$ is also compact. 

\begin{lem}\label{l1}
Let $k\in\mathbb{N}, (\xi_1,\ldots,\xi_k) \in \mathbb{R}^k$ and $E$ be a continuous $2k$-operator. Assume that for every choice of $(f_1,g_1,\ldots,f_k,g_k) \in \mathbb{A}_{\xi_1,\ldots,\xi_k}$ there exists a natural number $s$, such that
\begin{equation}\label{eqe}
 E(f_1,g_1,\ldots,f_k,g_k)(s) = 0.
\end{equation}
Then there exists $T\in\mathbb{N}$, such that for all $(f_1,g_1,\ldots,f_k,g_k) \in \mathbb{A}_{\xi_1,\ldots,\xi_k}$ the equality (\ref{eqe}) holds for some $s \leq T$.
\end{lem}
\begin{proof} On the set $\mathbb{A}_{\xi_1,\ldots,\xi_k}$ we define the functional $M$ by
$$ M(f_1,g_1,\ldots,f_k,g_k) = \mu s[E(f_1,g_1,\ldots,f_k,g_k)(s) = 0]. $$
The assumptions in the lemma guarantee that the minimization is always successful. Let us fix a $2k$-tuple of unary functions $(f_1^0,g_1^0,\ldots,f_k^0,g_k^0) \in \mathbb{A}_{\xi_1,\ldots,\xi_k}$ and let $$M(f_1^0,g_1^0,\ldots,f_k^0,g_k^0) = n.$$ By the continuity of $E$ we can choose a natural number $z$, such that for any $2k$-tuple of unary functions $(f_1,g_1,\ldots,f_k,g_k)$ the equalities
$$ f_i(x) = f_i^0(x),\;\;\; g_i(x) = g_i^0(x)\;\;\;(i=1,\ldots,k) $$
for all natural $x \leq z$ imply the equalities
$$ E(f_1,g_1,\ldots,f_k,g_k)(m) = E(f_1^0,g_1^0,\ldots,f_k^0,g_k^0)(m)\;\;\;\text{ for }m=0,1,\ldots,n. $$
But from the last equalities we obtain $M(f_1,g_1,\ldots,f_k,g_k) = n = M(f_1^0,g_1^0,\ldots,f_k^0,g_k^0)$. In other words, on the open neighbourhood $U_0 = \mathbb{F}^{-1}[V_0]$ of $(f_1^0,g_1^0,\ldots,f_k^0,g_k^0)$, where
\begin{multline*}
 V_0 = \{(f_1^0(0),g_1^0(0),\ldots,f_k^0(0),g_k^0(0))\} \times \ldots \\
    \times \{(f_1^0(z),g_1^0(z),\ldots,f_k^0(z),g_k^0(z))\} \\
      \times \mathbb{A}_{\xi_1,\ldots,\xi_k}^{[z+1]} \times \mathbb{A}_{\xi_1,\ldots,\xi_k}^{[z+2]} \times \ldots,
\end{multline*}
the functional $M$ assumes constant value $n$. By compactness, the set $\mathbb{A}_{\xi_1,\ldots,\xi_k}$ can be covered with finitely many open sets, such that the functional $M$ is constant on each one of them. Therefore, $M$ assumes only finitely many values and we can choose $T$ to be the largest of these values.
\end{proof}

To make use of our new restricted system of names for tuples of real numbers, we need a uniform way to obtain this names. For this purpose, we define the auxiliary function $ehelp : \mathbb{N}^4 \rightarrow \mathbb{N}$ by
$$ ehelp(p,q,r,n) = \left\lfloor (n+1)\frac{p \dotminus q}{r+1} + \frac{1}{2} \right\rfloor. $$
It is clear that $ehelp \in \mathcal{M}^2$. Its main properties are listed in the following remark.

\begin{rem}\label{propehelp}
 For all natural numbers $p,q,r,n$ we have
  $$ ehelp(p,q,r,n).ehelp(q,p,r,n) = 0 $$
and the inequality
  $$ \left|\frac{ehelp(p,q,r,n) - ehelp(q,p,r,n)}{n+1} - \frac{p-q}{r+1} \right| \leq \frac{1}{2(n+1)}. $$
\end{rem}

The transition to the special names is realized by the $3$-operator $K$ from Section 1.3 of \cite{m2}, defined by
 $$ K(f,g,h)(n) = ehelp(f(2n+1),g(2n+1),h(2n+1),n). $$
Of course, $K$ is $\mathcal{M}^2$-substitutional ($K\in\mathbf{O}_{\mathcal{M}^2}$).

\begin{lem}\label{operatorK}
For all unary functions $f,g,h$ and natural numbers $n$ at least one of the numbers $K(f,g,h)(n)$ and $K(g,f,h)(n)$ is $0$. If $(f,g,h)$ names a real number $\xi$, then the triple $(K(f,g,h),K(g,f,h),\mathrm{id}_\mathbb{N})$ also names $\xi$.
\end{lem}

\begin{cor}\label{coroperK}
For any $k$ and $(\xi_1,\ldots,\xi_k) \in \mathbb{R}^k$, if $(f_i,g_i,h_i)$ names $\xi_i$ for $i=1,\ldots,k$, then
$$ (K(f_1,g_1,h_1),K(g_1,f_1,h_1),\ldots,K(f_k,g_k,h_k),K(g_k,f_k,h_k)) $$
belongs to $\mathbb{A}_{\xi_1,\ldots,\xi_k}$.
\end{cor}

The proof of Lemma \ref{operatorK} is based on Remark \ref{propehelp} and can be found in Section 1.3 of \cite{m2}. Corollary \ref{coroperK} follows immediately.

\begin{lem}\label{l2}
Let $\mathbf{O}$ be a class of operators, which is the second component of some acceptable pair $(\mathcal{F},\mathbf{O})$. Let $k\in\mathbb{N}$ and $\theta : D \rightarrow \mathbb{R}, D \subseteq \mathbb{R}^k$ be a conditionally $\mathbf{O}$-computable real function. Then there exist $3k$-operator $E$ and $(3k+1)$-operators $F, G, H$, all of them belonging to $\mathbf{O}$, such that for any point $(\xi_1,\ldots,\xi_k) \in D$ there exists $T\in\mathbb{N}$, for which the following conditions hold:
\begin{enumerate}
\item If the triples $(f_1,g_1,h_1),\ldots,(f_k,g_k,h_k)$ name $\xi_1,\ldots,\xi_k$ respectively, then the equality
\begin{equation}\label{eqeful}
 E(f_1,g_1,h_1,\ldots,f_k,g_k,h_k)(s) = 0
\end{equation}
holds for some natural number $s \leq T$.
\item Whenever $(f_1,g_1,h_1),\ldots,(f_k,g_k,h_k)$ name $\xi_1,\ldots,\xi_k$ respectively and the equality (\ref{eqeful}) holds for some natural number $s$, the triple
$$ (F(f_1,g_1,h_1,\ldots,f_k,g_k,h_k,\widehat{s}), $$
$$ G(f_1,g_1,h_1,\ldots,f_k,g_k,h_k,\widehat{s}), $$
$$ H(f_1,g_1,h_1,\ldots,f_k,g_k,h_k,\widehat{s})) $$
names the real number $\theta(\xi_1,\ldots,\xi_k)$.
\end{enumerate}
\end{lem}
\begin{proof} Let $E,F,G,H\in\mathbf{O}$ be witnesses for $\theta$ from Definition \ref{defcondo}. We define $3k$-operator $E'$ and $(3k+1)$-operators $F',G',H'$ by
\begin{align*}
 E'(f_1,g_1,h_1,\ldots,f_k,g_k,h_k) & \\
       = E(K(f_1,g_1,h_1),K(g_1,&f_1,h_1),\mathrm{id}_\mathbb{N},\ldots,K(f_k,g_k,h_k),K(g_k,f_k,h_k),\mathrm{id}_\mathbb{N}),\\
 F'(f_1,g_1,h_1,\ldots,f_k,g_k,h_k,a) & \\
       = F(K(f_1,g_1,h_1),K(g_1,&f_1,h_1),\mathrm{id}_\mathbb{N},\ldots,K(f_k,g_k,h_k),K(g_k,f_k,h_k),\mathrm{id}_\mathbb{N},a),\\
 G'(f_1,g_1,h_1,\ldots,f_k,g_k,h_k,a) & \\
       = G(K(f_1,g_1,h_1),K(g_1,&f_1,h_1),\mathrm{id}_\mathbb{N},\ldots,K(f_k,g_k,h_k),K(g_k,f_k,h_k),\mathrm{id}_\mathbb{N},a),\\
 H'(f_1,g_1,h_1,\ldots,f_k,g_k,h_k,a) & \\
        = H(K(f_1,g_1,h_1),K(g_1,&f_1,h_1),\mathrm{id}_\mathbb{N},\ldots,K(f_k,g_k,h_k),K(g_k,f_k,h_k),\mathrm{id}_\mathbb{N},a).
\end{align*}
We will show that we can choose $E',F',G',H'$ in the role of $E,F,G,H$ from the statement of the lemma. We have $K \in \mathbf{O}$, since $\mathbf{O}_{\mathcal{M}^2} \subseteq \mathbf{O}_\mathcal{F} \subseteq \mathbf{O}$. From conditions (4), (5) and (7) of Definition \ref{accpair} we obtain that $E',F',G',H' \in \mathbf{O}$. Let us fix a point $(\xi_1,\ldots,\xi_k) \in D$.

We define the $2k$-operator $E''$ by
$$ E''(f_1,g_1,\ldots,f_k,g_k) = E(f_1,g_1,\mathrm{id}_\mathbb{N},\ldots,f_k,g_k,\mathrm{id}_\mathbb{N}). $$
It is clear that $E''$ is continuous, since $E$ is continuous from condition (3) of Definition \ref{accpair}. Moreover, if $(f_1,g_1,\ldots,f_k,g_k) \in \mathbb{A}_{\xi_1,\ldots,\xi_k}$, then $(f_i,g_i,\mathrm{id}_\mathbb{N})$ names $\xi_i$ for $i=1,\ldots,k$ and condition (1) from Definition \ref{defcondo} gives that there exists a natural number $s$, such that
$$ E(f_1,g_1,\mathrm{id}_\mathbb{N},\ldots,f_k,g_k,\mathrm{id}_\mathbb{N})(s) = E''(f_1,g_1,\ldots,f_k,g_k)(s) = 0. $$
We apply Lemma \ref{l1} for the operator $E''$ and choose the corresponding natural number $T$. Using Corollary \ref{coroperK} it is easily seen that the same number $T$ satisfies condition (1) in the present lemma (with $E'$ in the role of $E$). Condition (2) is an easy consequence of Lemma \ref{operatorK} and condition (2) from Definition \ref{defcondo}.
\end{proof}

\section{The characterization theorem}

Now we are ready to formulate a proper notion of conditional computability for real functions in the style of Tent and Ziegler and prove the characterization theorem.

\begin{defi}\label{tzcond} For a class of functions $\mathcal{F}$ and $k \in \mathbb{N}$ the real function $\theta : D \rightarrow \mathbb{R}, D \subseteq \mathbb{R}^k$ will be called \emph{conditionally $\mathcal{F}$-computable in the style of Tent and Ziegler}, if there exist functions $d_0 \in \mathcal{T}_1 \cap \mathcal{F}, d \in \mathcal{T}_2 \cap \mathcal{F}, e \in \mathcal{T}_{3k+1}\cap\mathcal{F}$ and $f,g,h \in \mathcal{T}_{6k+2} \cap\mathcal{F}$, such that for all $(\xi_1,\ldots,\xi_k) \in D$ we have:
\begin{enumerate}
\item There exists $s_0\in\mathbb{N}$, such that for all natural numbers $s \geq s_0$ and

all $p_1^0,q_1^0,r_1^0,\ldots,p_k^0,q_k^0,r_k^0\in\mathbb{N}$, the inequalities
\begin{equation}\label{ccteq1}
 \left|\frac{p_i^0 - q_i^0}{r_i^0 + 1} - \xi_i\right| < \frac{1}{d_0(s)+1}\;\;\;\text{ for } i = 1, \ldots, k
\end{equation}
imply the equality
\begin{equation}\label{ccteq2}
 e(p_1^0,q_1^0,r_1^0,\ldots,p_k^0,q_k^0,r_k^0,s) = 0.
\end{equation}
\item For all natural numbers $s, p_1^0,q_1^0,r_1^0,\ldots,p_k^0,q_k^0,r_k^0, p_1,q_1,r_1,\ldots,p_k,q_k,r_k, t$, which satisfy (\ref{ccteq1}), (\ref{ccteq2}) and 
\begin{equation}\label{ccteq3}
 \left|\xi_i\right| \leq s+1,\;\;\; \left|\frac{p_i - q_i}{r_i + 1} - \xi_i\right| < \frac{1}{d(s,t)+1}\;\;\;\text{ for }i = 1,\ldots,k,
\end{equation}
we have
\begin{equation}\label{ccteq4}
 \left|\frac{p - q}{r + 1} - \theta(\xi_1,\ldots,\xi_k)\right| < \frac{1}{t+1},
\end{equation}
where
\begin{align}
 p = f(p_1^0,q_1^0,r_1^0,\ldots,p_k^0,q_k^0,r_k^0, p_1,q_1,r_1,\ldots,p_k,q_k,r_k, s, t), \label{ccteq51} \\
 q = g(p_1^0,q_1^0,r_1^0,\ldots,p_k^0,q_k^0,r_k^0, p_1,q_1,r_1,\ldots,p_k,q_k,r_k, s, t), \label{ccteq52} \\
 r = h(p_1^0,q_1^0,r_1^0,\ldots,p_k^0,q_k^0,r_k^0, p_1,q_1,r_1,\ldots,p_k,q_k,r_k, s, t). \label{ccteq53}
\end{align}
\end{enumerate}
\end{defi}

\noindent The intuition behind this definition is not very clear at first glance, since its main purpose is to be proper for proving the characterization thereom. The important thing is that the definition does not refer to names of real numbers.

\begin{thm}\label{charthcond} Let $(\mathcal{F},\mathbf{O})$ be an acceptable pair and $k\in\mathbb{N}$. The real function $\theta : D \rightarrow \mathbb{R}, D \subseteq \mathbb{R}^k$ is conditionally $\mathbf{O}$-computable if and only if $\theta$ is conditionally $\mathcal{F}$-computable in the style of Tent and Ziegler.
\end{thm}
\begin{proof} Since $\mathcal{F}$ satisfies conditions (1) and (2) from Definition \ref{accpair}, we have that $\mathcal{M}^2 \subseteq \mathcal{F}$ and $\mathcal{F}$ is closed under bounded minimum and bounded maximum, as noted after Definition \ref{accpair}. We assume $k=1$. The generalization of the proof to arbitrary $k$ is straightforward.

For the direction $(\Longleftarrow)$, suppose that $d_0,d,e,f,g,h\in\mathcal{F}$ are functions, which satisfy the requirements from Definition \ref{tzcond}. We will show that $\theta$ is conditionally $\mathbf{O}$-computable. We define the operator $E$ by
$$ E(f_1,g_1,h_1)(s') = e(f_1(d_0(s)),g_1(d_0(s)),h_1(d_0(s)),s), $$
where $s = \max(f_1(0),g_1(0),s')$. We also define operators $F, G, H$ by
$$ F(f_1,g_1,h_1,a)(t) = p,\;\;\; G(f_1,g_1,h_1,a)(t) = q,\;\;\; H(f_1,g_1,h_1,a)(t) = r, $$
where the numbers $p, q, r$ are defined by the equalities (\ref{ccteq51}), (\ref{ccteq52}), (\ref{ccteq53}) with
\begin{align}
 p_1^0 = f_1(d_0(s)),\;\;\;q_1^0 = g_1(d_0(s)),\;\;\;r_1^0 = h_1(d_0(s)), \label{eqpqr1}\\
 p_1 = f_1(d(s,t)),\;\;\;q_1 = g_1(d(s,t)),\;\;\;r_1 = h_1(d(s,t)) \label{eqpqr2}
\end{align}
and $s = \max(f_1(0),g_1(0),a(t))$. From conditions (4), (5) and (6) in Definition \ref{accpair} we obtain $E,F,G,H \in \mathbf{O}$. We will show that these operators can be chosen as witnesses for the conditional $\mathbf{O}$-computability of $\theta$. Let $\xi_1 \in D$ and $(f_1,g_1,h_1)$ name $\xi_1$.
Let us choose $s_0$ from condition (1) of Definition \ref{tzcond}. Let $s = \max(f_1(0),g_1(0),s_0)$. We have that $s \geq s_0$ and (\ref{ccteq1}) is satisfied for the numbers $p_1^0, q_1^0, r_1^0$, defined by the equalities (\ref{eqpqr1}). Then from (\ref{ccteq2}) we obtain
$$ e(p_1^0, q_1^0, r_1^0, s) = 0 = E(f_1,g_1,h_1)(s_0). $$
Thus there exists $s'$, such that $E(f_1,g_1,h_1)(s') = 0$.

Now let $s'\in\mathbb{N}$ and $E(f_1,g_1,h_1)(s') = 0$, that is we have
$$ e(f_1(d_0(s)),g_1(d_0(s)),h_1(d_0(s)),s) = 0, $$ 
where $s = \max(f_1(0),g_1(0),s')$. We use condition (2) of Definition \ref{tzcond} for arbitrary $t$, the chosen $s$ and the numbers $p_1^0,q_1^0,r_1^0,p_1,q_1,r_1$, defined by the equalities (\ref{eqpqr1}) and (\ref{eqpqr2}). The premises (\ref{ccteq1}), (\ref{ccteq2}), (\ref{ccteq3}) are satisfied, since $(f_1,g_1,h_1)$ names $\xi_1$ and therefore
$$ \left|\xi_1\right| < \left|f_1(0) - g_1(0)\right| + 1 \leq \max(f_1(0),g_1(0)) + 1 \leq s+1. $$ 
We obtain the inequality (\ref{ccteq4}) for the numbers $p, q, r$, defined by the equalities (\ref{ccteq51}), (\ref{ccteq52}), (\ref{ccteq53}). But for $a = \lambda x.s' = \widehat{s'}$ we have
$$ p = F(f_1,g_1,h_1,\widehat{s'})(t),\;\;\; q = G(f_1,g_1,h_1,\widehat{s'})(t),\;\;\; r = H(f_1,g_1,h_1,\widehat{s'})(t). $$
Thus
$$ \left|\frac{F(f_1,g_1,h_1,\widehat{s'})(t) - G(f_1,g_1,h_1,\widehat{s'})(t)}{H(f_1,g_1,h_1,\widehat{s'})(t) + 1} - \theta(\xi_1)\right| < \frac{1}{t+1} $$
for all $t \in \mathbb{N}$, that is
$$ (F(f_1,g_1,h_1,\widehat{s'}), G(f_1,g_1,h_1,\widehat{s'}), H(f_1,g_1,h_1,\widehat{s'})) $$
is a name of $\theta(\xi_1)$. Therefore $\theta$ is conditionally $\mathbf{O}$-computable.

For the converse $(\Longrightarrow)$, let us suppose that $\theta$ is conditionally $\mathbf{O}$-computable. Since $\mathbf{O}$ is the second component of an acceptable pair, we can apply Lemma \ref{l2} and choose operators $E,F,G,H \in \mathbf{O}$ for the real function $\theta$, according to the statement of the lemma. We then choose $1$-operators $\Omega,\Omega_1,\Omega_2,\Omega_3 \in \mathbf{O}$ from the uniformity condition in Definition \ref{accpair}, corresponding to the operators $E,F,G,H$, respectively. We define consecutively the following functions:
$$ u(x,s) = (s+2)(x+1), $$
$$ v(s,y) = \Omega(\lambda x.u(x,s))(y), $$
$$ v'(s) = \max_{y \leq s} v(s,y), $$
$$ d_0(s) = 6v'(s) + 5, $$
$$ w(s,t) = \max \bigl(\Omega_1(\lambda x.u(x,s))(t), \Omega_2(\lambda x.u(x,s))(t), \Omega_3(\lambda x.u(x,s))(t)\bigr), $$
$$ w'(s,t) = \max(v'(s), w(s,t)), $$
$$ d(s,t) = 6w'(s,t) + 5, $$
$$ b(p_1^0,q_1^0,r_1^0, s) = \mu_{x \leq s}[E(f_1^0,g_1^0,\mathrm{id}_{\mathbb{N}})(x) = 0], $$
$$ e(p_1^0,q_1^0,r_1^0, s) = \min_{x \leq s}E(f_1^0,g_1^0,\mathrm{id}_{\mathbb{N}})(x), $$
$$ f(p_1^0,q_1^0,r_1^0,p_1,q_1,r_1,s,t) = F(f_1,g_1,\mathrm{id}_{\mathbb{N}},\lambda x.b(p_1^0,q_1^0,r_1^0,s))(t), $$
$$ g(p_1^0,q_1^0,r_1^0,p_1,q_1,r_1,s,t) = G(f_1,g_1,\mathrm{id}_{\mathbb{N}},\lambda x.b(p_1^0,q_1^0,r_1^0,s))(t), $$
$$ h(p_1^0,q_1^0,r_1^0,p_1,q_1,r_1,s,t) = H(f_1,g_1,\mathrm{id}_{\mathbb{N}},\lambda x.b(p_1^0,q_1^0,r_1^0,s))(t), $$
where
\begin{equation}\label{defzero}
f_1^0 = \lambda x.ehelp(p_1^0,q_1^0,r_1^0,x),\;\;\; g_1^0 = \lambda x.ehelp(q_1^0,p_1^0,r_1^0,x),
\end{equation}
and
\begin{equation}\label{deff}
f_1 = \lambda x.\begin{cases}
               ehelp(p_1^0,q_1^0,r_1^0,x), \text{ if } x \leq v'(s),\\
							 ehelp(p_1,q_1,r_1,x), \text{ if } x > v'(s),
						\end{cases}
\end{equation}
\begin{equation}\label{defg}
g_1 = \lambda x.\begin{cases}
               ehelp(q_1^0,p_1^0,r_1^0,x), \text{ if } x \leq v'(s),\\
							 ehelp(q_1,p_1,r_1,x), \text{ if } x > v'(s).
						\end{cases}
\end{equation}
Here we use the function $ehelp$ defined in the previous section, $ehelp \in \mathcal{F}$. It is not hard to see that all of the defined functions belong to the class $\mathcal{F}$ (due to conditions (1), (2) and (8) in Definition \ref{accpair}). Our purpose is to show that $d_0,d,e,f,g,h$ satisfy the two conditions in Definition \ref{tzcond}. We fix $\xi_1 \in D$.

Let us choose a natural number $T$ according to Lemma \ref{l2}. Let $s_0$ be a natural number, such that $s_0 \geq T$ and $|\xi_1| \leq s_0+1$. We claim that condition (1) in Definition \ref{tzcond} is satisfied for this $s_0$. Indeed, let $s \geq s_0$ and $p_1^0,q_1^0,r_1^0\in\mathbb{N}$ are numbers, such that (\ref{ccteq1}) holds. We define the unary functions $f_1^0$ and $g_1^0$ by the equalities (\ref{defzero}). From Remark \ref{propehelp} for all $x\in\mathbb{N}$ we have
\begin{equation}\label{zeroineq}
 \left| \frac{f_1^0(x) - g_1^0(x)}{x+1} - \frac{p_1^0 - q_1^0}{r_1^0 + 1}\right| \leq \frac{1}{2(x+1)},\;\;\; f_1^0(x).g_1^0(x) = 0.
\end{equation}
Let us choose unary functions $\overline{f_1^0}$ and $\overline{g_1^0}$, such that for all $x\in\mathbb{N}$
$$ x \leq v'(s) \Longrightarrow \overline{f_1^0}(x) = f_1^0(x)\;\&\; \overline{g_1^0}(x) = g_1^0(x), $$
\begin{equation}\label{zeroname}
 x > v'(s) \Longrightarrow \left| \frac{\overline{f_1^0}(x) - \overline{g_1^0}(x)}{x+1} - \xi_1\right| < \frac{1}{x+1}\;\&\; \overline{f_1^0}(x).\overline{g_1^0}(x) = 0.
\end{equation}
The right-hand side of (\ref{zeroname}) is true also for $x \leq v'(s)$, since
$$ \left| \frac{\overline{f_1^0}(x) - \overline{g_1^0}(x)}{x+1} - \xi_1\right| = \left| \frac{f_1^0(x) - g_1^0(x)}{x+1} - \frac{p_1^0 - q_1^0}{r_1^0 + 1} + \frac{p_1^0 - q_1^0}{r_1^0 + 1} - \xi_1\right| $$
$$ < \frac{1}{2(x+1)} + \frac{1}{d_0(s)+1} \leq \frac{1}{2(x+1)} + \frac{1}{2(v'(s)+1)} \leq \frac{1}{2(x+1)} + \frac{1}{2(x+1)} = \frac{1}{x+1}, $$
$$ \overline{f_1^0}(x).\overline{g_1^0}(x) = f_1^0(x).g_1^0(x) = 0. $$
We used that $d_0(s) \geq 2v'(s)+1$. Thus $(\overline{f_1^0}, \overline{g_1^0}, \mathrm{id}_\mathbb{N})$ names $\xi_1$. According to the choice of $T$ there exists $s_1$, such that $s_1 \leq T$ and $E(\overline{f_1^0},\overline{g_1^0},\mathrm{id}_\mathbb{N})(s_1) = 0$. Obviously, $\mathrm{id}_\mathbb{N}$ is majorized by $\lambda x.u(x,s)$. From the right-hand side of (\ref{zeroname}) it follows that any of the numbers $\overline{f_1^0}(x), \overline{g_1^0}(x)$ is less than $(|\xi_1|+1)(x+1)$. But from the choice of $s_0$ we have $|\xi_1| \leq s_0+1 \leq s+1$ and so the functions $\overline{f_1^0},\overline{g_1^0}$ are majorized by $\lambda x.u(x,s)$. The same is true for the functions $f_1^0,g_1^0$, since from (\ref{zeroineq}) and (\ref{ccteq1}) we have
$$ \left|\frac{f_1^0(x) - g_1^0(x)}{x+1}\right| \leq \left|\frac{p_1^0-q_1^0}{r_1^0+1}\right| + \frac{1}{2} < |\xi_1| + 1. $$
For $x \leq v(s,s_1)$ we have the equalities $f_1^0(x) = \overline{f_1^0}(x),\;g_1^0(x) = \overline{g_1^0}(x)$, since $s_1 \leq T \leq s_0 \leq s$ and therefore $v(s,s_1) \leq v'(s)$. From the uniformity condition and the choice of $\Omega$ we obtain
$$ E(f_1^0,g_1^0,\mathrm{id}_\mathbb{N})(s_1) = E(\overline{f_1^0},\overline{g_1^0},\mathrm{id}_\mathbb{N})(s_1) = 0. $$
Thus $e(p_1^0,q_1^0,r_1^0,s) = \min_{x \leq s} E(f_1^0,g_1^0,\mathrm{id}_\mathbb{N})(x) = 0$, because $s_1 \leq s$.

Now we prove condition (2) from Definition \ref{tzcond}. Let $s,p_1^0,q_1^0,r_1^0,p_1,q_1,r_1,t$ be natural numbers, which satisfy the premises (\ref{ccteq1}), (\ref{ccteq2}), (\ref{ccteq3}) in condition (2). In other words, we have the inequalities
$$ \left|\frac{p_1^0 - q_1^0}{r_1^0 + 1} - \xi_1\right| < \frac{1}{d_0(s)+1},\;\;\; \left|\xi_1\right| \leq s+1,\;\;\; \left|\frac{p_1 - q_1}{r_1 + 1} - \xi_1\right| < \frac{1}{d(s,t)+1} $$ and the equality $ e(p_1^0,q_1^0,r_1^0,s) = 0 $.
Let $p,q,r$ be defined by the equalities (\ref{ccteq51}), (\ref{ccteq52}), (\ref{ccteq53}). Our goal is to prove the inequality (\ref{ccteq4}).
First we define the functions $f_1^0$ and $g_1^0$ by the equalities (\ref{defzero}). For all $x\in\mathbb{N}$ we have (\ref{zeroineq}) from above. We also define the functions $f_1$, $g_1$ with the equalities (\ref{deff}), (\ref{defg}), respectively. We note explicitly that for $x \leq v'(s)$ we have
$$ f_1^0(x) = f_1(x),\;\;\; g_1^0(x) = g_1(x). $$
For all $x\in\mathbb{N}$ we have
$$ f_1(x).g_1(x) = 0, $$
\begin{equation}\label{ineqfg1}
 x \leq v'(s) \Longrightarrow \left| \frac{f_1(x) - g_1(x)}{x+1} - \frac{p_1 - q_1}{r_1 + 1}\right| \leq \frac{5}{6(x+1)},\\
\end{equation}
\begin{equation}\label{ineqfg2}
 x > v'(s) \Longrightarrow \left| \frac{f_1(x) - g_1(x)}{x+1} - \frac{p_1 - q_1}{r_1 + 1}\right| \leq \frac{1}{2(x+1)}.
\end{equation}
Only (\ref{ineqfg1}) requires a proof, (\ref{ineqfg2}) follows from Remark \ref{propehelp}. Let $x \leq v'(s)$. Then
$$ \left| \frac{f_1(x) - g_1(x)}{x+1} - \frac{p_1 - q_1}{r_1 + 1}\right| $$
$$ = \left| \frac{f_1^0(x) - g_1^0(x)}{x+1} - \frac{p_1^0 - q_1^0}{r_1^0 + 1} + \frac{p_1^0 - q_1^0}{r_1^0 + 1} - \xi_1 + \xi_1 - \frac{p_1 - q_1}{r_1 + 1}\right| $$
$$ < \frac{1}{2(x+1)} + \frac{1}{d_0(s)+1} + \frac{1}{d(s,t)+1} \leq \frac{1}{2(x+1)} + \frac{1}{6(v'(s)+1)} + \frac{1}{6(v'(s)+1)} $$
$$ \leq \frac{1}{2(x+1)} + \frac{1}{6(x+1)} + \frac{1}{6(x+1)} = \frac{5}{6(x+1)}. $$
Here we used that $d_0(s) = 6v'(s) + 5$ and $d(s,t) \geq 6v'(s) + 5$.

From $e(p_1^0,q_1^0,r_1^0,s) = 0$ we have that $E(f_1^0,g_1^0,\mathrm{id}_\mathbb{N})(s_1) = 0$ for some $s_1 \leq s$. Let us choose the least such $s_1$.
We choose functions $\overline{f_1^0}$ and $\overline{g_1^0}$ exactly as above and we obtain that $(\overline{f_1^0}, \overline{g_1^0}, \mathrm{id}_\mathbb{N})$ names $\xi_1$ and
$$ E(\overline{f_1^0},\overline{g_1^0},\mathrm{id}_\mathbb{N})(s_1) = E(f_1^0,g_1^0,\mathrm{id}_\mathbb{N})(s_1) = 0. $$
The only difference is that this time we have the inequality $|\xi_1| \leq s+1$ directly in our premises.
We then choose functions $\overline{f_1}$ and $\overline{g_1}$, such that
$$ x \leq w'(s,t) \Longrightarrow \overline{f_1}(x) = f_1(x)\;\&\; \overline{g_1}(x) = g_1(x), $$
$$ x > w'(s,t) \Longrightarrow \left| \frac{\overline{f_1}(x) - \overline{g_1}(x)}{x+1} - \xi_1\right| < \frac{1}{x+1}\;\&\; \overline{f_1}(x).\overline{g_1}(x) = 0. $$
Then for $x \leq w'(s,t)$ we have
$$ \overline{f_1}(x).\overline{g_1}(x) = f_1(x).g_1(x) = 0, $$
$$ x \leq v'(s) \Longrightarrow \left| \frac{\overline{f_1}(x) - \overline{g_1}(x)}{x+1} - \xi_1\right| = \left| \frac{f_1(x) - g_1(x)}{x+1} - \frac{p_1 - q_1}{r_1 + 1} + \frac{p_1 - q_1}{r_1 + 1} - \xi_1\right| $$
$$ < \frac{5}{6(x+1)} + \frac{1}{d(s,t)+1} \leq \frac{5}{6(x+1)} + \frac{1}{6(v'(s)+1)} \leq \frac{5}{6(x+1)} + \frac{1}{6(x+1)} = \frac{1}{x+1}, $$
$$ x > v'(s) \Longrightarrow \left| \frac{\overline{f_1}(x) - \overline{g_1}(x)}{x+1} - \xi_1\right| = \left| \frac{f_1(x) - g_1(x)}{x+1} - \frac{p_1 - q_1}{r_1 + 1} + \frac{p_1 - q_1}{r_1 + 1} - \xi_1\right| $$
$$ < \frac{1}{2(x+1)} + \frac{1}{d(s,t)+1} \leq \frac{1}{2(x+1)} + \frac{1}{2(x+1)} = \frac{1}{x+1}. $$
For the first implication we used (\ref{ineqfg1}) and the inequality $d(s,t) \geq 6v'(s) + 5$ and for the second one -- (\ref{ineqfg2}) and $d(s,t) \geq 2w'(s,t) + 1 \geq 2x + 1$. Thus we obtain that ($\overline{f_1},\overline{g_1},\mathrm{id}_\mathbb{N}$) names $\xi_1$. We have that $\mathrm{id}_\mathbb{N}, \overline{f_1^0}, \overline{g_1^0}, \overline{f_1}, \overline{g_1}$ are majorized by $\lambda x.u(x,s)$ and also
$$ \overline{f_1^0}(x) = f_1^0(x) = f_1(x) = \overline{f_1}(x),\;\;\;\overline{g_1^0}(x) = g_1^0(x) = g_1(x) = \overline{g_1}(x) $$ 
for $x \leq v(s,s_1)$, since $v(s,s_1) \leq v'(s) \leq w'(s,t)$.
From the uniformity condition and the choice of $\Omega$,
$$ E(\overline{f_1},\overline{g_1},\mathrm{id}_\mathbb{N})(s_1) = E(\overline{f_1^0},\overline{g_1^0},\mathrm{id}_\mathbb{N})(s_1) = 0. $$
The choice of the operators $E, F, G, H$ gives the inequality
$$ \left| \frac{F(\overline{f_1},\overline{g_1}, \mathrm{id}_\mathbb{N},\lambda x.s_1)(t) - G(\overline{f_1},\overline{g_1}, \mathrm{id}_\mathbb{N},\lambda x.s_1)(t)}{H(\overline{f_1},\overline{g_1}, \mathrm{id}_\mathbb{N}, \lambda x.s_1)(t) + 1} - \theta(\xi_1) \right| < \frac{1}{t+1}. $$
But $\mathrm{id}_\mathbb{N}, \lambda x.s_1, \overline{f_1}, \overline{g_1}$ are majorized by $\lambda x.u(x,s)$, because $s_1 \leq s$. The same is true for $f_1,g_1$, since
$$ \left|\frac{f_1(x) - g_1(x)}{x+1}\right| \leq \left|\frac{p_1-q_1}{r_1+1}\right| + \frac{5}{6} < |\xi_1| + 1 \leq s+2. $$
We also have $\overline{f_1}(x) = f_1(x),\; \overline{g_1}(x) = g_1(x)$ for $x \leq w(s,t)$, since $w(s,t) \leq w'(s,t)$. From the uniformity condition and the choice of the operators $\Omega_1,\Omega_2,\Omega_3$ we obtain
$$ F(\overline{f_1},\overline{g_1}, \mathrm{id}_\mathbb{N},\lambda x.s_1)(t) = F(f_1,g_1,\mathrm{id}_\mathbb{N},\lambda x.s_1)(t), $$
$$ G(\overline{f_1},\overline{g_1}, \mathrm{id}_\mathbb{N},\lambda x.s_1)(t) = G(f_1,g_1,\mathrm{id}_\mathbb{N},\lambda x.s_1)(t), $$
$$ H(\overline{f_1},\overline{g_1}, \mathrm{id}_\mathbb{N},\lambda x.s_1)(t) = H(f_1,g_1,\mathrm{id}_\mathbb{N},\lambda x.s_1)(t). $$
Moreover, $s_1 = \mu_{x \leq s}[ E(f_1^0,g_1^0,\mathrm{id}_\mathbb{N})(x) = 0 ] = b(p_1^0,q_1^0,r_1^0,s)$, thus
$$ F(f_1,g_1,\mathrm{id}_\mathbb{N},\lambda x.s_1)(t) = F(f_1,g_1,\mathrm{id}_\mathbb{N},\lambda x.b(p_1^0,q_1^0,r_1^0,s))(t) = p, $$
$$ G(f_1,g_1,\mathrm{id}_\mathbb{N},\lambda x.s_1)(t) = G(f_1,g_1,\mathrm{id}_\mathbb{N},\lambda x.b(p_1^0,q_1^0,r_1^0,s))(t) = q, $$
$$ H(f_1,g_1,\mathrm{id}_\mathbb{N},\lambda x.s_1)(t) = H(f_1,g_1,\mathrm{id}_\mathbb{N},\lambda x.b(p_1^0,q_1^0,r_1^0,s))(t) = r. $$
We reached our goal -- the inequality (\ref{ccteq4}) and so the proof is completed.
\end{proof}

\begin{cor} Let $\mathcal{F}$ be a class of functions, which satisfies conditions (1) and (2) from Definition \ref{accpair}. Then a real function is conditionally $\mathbf{O}_\mathcal{F}$-computable if and only if it is conditionally $\mathcal{F}$-computable in the style of Tent and Ziegler.
\end{cor}
\begin{proof} We apply Proposition \ref{propac}. \end{proof}

\begin{cor}\label{corcon} For an acceptable pair $(\mathcal{F},\mathbf{O})$ and a real function $\theta$ the following are equivalent:
	\begin{itemize}
	 \item $\theta$ is conditionally $\mathbf{O}$-computable,
	 \item $\theta$ is conditionally $\mathcal{F}$-computable in the style of Tent and Ziegler,
	 \item $\theta$ is conditionally $\mathbf{O}_\mathcal{F}$-computable.
  \end{itemize}
	\end{cor}

\section*{Conclusion}
We plan to extend our results for uniform and conditional $\mathcal{M}^2$-computability to other nonelementary real functions from calculus, such as the gamma function and the Riemann zeta function. The results of Tent and Ziegler \cite{tz} in this direction, particularly Theorem 5.1 for the complexity of integration, cannot be used for the class $\mathcal{M}^2$, because it is not known if $\mathcal{M}^2$ is closed under bounded summation.

Corollary \ref{corcon} will prove to be very useful, because it provides extra machinery for proving conditional $\mathcal{M}^2$-computability. Namely, we can take any class of operators $\mathbf{O}$, which forms an acceptable pair with $\mathcal{M}^2$ and prove conditional $\mathbf{O}$-computability. And such classes of operators $\mathbf{O}$, which are broader than the class of $\mathcal{M}^2$-substitutional operators, do exist.

\end{document}